\documentclass{article}
\usepackage{amsmath,amssymb}
\usepackage{color}
\usepackage{rotating}
\usepackage{lscape}
\usepackage{tablefootnote}
\usepackage[square,numbers,sort]{natbib}

\usepackage{bm}
\usepackage{bbm}

\usepackage{threeparttable}

\usepackage[font=small,margin=0pt,justification=centering]{caption}

\usepackage{agor}
\usepackage{algorithm}
\usepackage{algpseudocode}

\newcommand{\pinv}[2]{\ensuremath{{\bm{{#1}_{#2}^\star}}}}

\newcommand{\leqnomode}{\tagsleft@true\let\veqno\@@leqno}
\newcommand{\reqnomode}{\tagsleft@false\let\veqno\@@eqno}

\def\SingleSpacedXI{\linespread{1.05}}
\SingleSpacedXI

\title{Polyhedral Analysis of Quadratic Optimization Problems with Stieltjes Matrices and Indicators}

\author{ Peijing Liu \thanks{{Department of Industrial and System Engineering, University of Southern California, \texttt{peijingl@usc.edu}.}} \and Alper Atamt\"urk \thanks{{Department of Industrial Engineering and Operations Research, University of California, Berkeley, \texttt{atamturk@berkeley.edu}.}} \and  Andr\'es G\'omez \thanks{{Department of Industrial and System Engineering, University of Southern California, \texttt{gomezand@usc.edu}.}} \and Simge K\"u\c{c}\"ukyavuz \thanks{{Department of Industrial Engineering and Management Sciences, Northwestern University, \texttt{simge@northwestern.edu}.}}}

\RequirePackage[normalem]{ulem} 
\RequirePackage{color}\definecolor{RED}{rgb}{1,0,0}\definecolor{BLUE}{rgb}{0,0,1}

\begin{document}
\maketitle
\begin{abstract} 
	In this paper, we consider convex quadratic optimization problems with indicators on the continuous variables. In particular, we assume that the Hessian of the quadratic term is a Stieltjes matrix, which naturally appears in sparse graphical inference problems and others. We describe an explicit convex formulation for the problem by studying the Stieltjes polyhedron arising as part of an extended formulation and exploiting the supermodularity of a set function defined on its extreme points.  Our computational results confirm that the proposed convex relaxation provides an exact optimal solution and may be an effective alternative, especially for instances with large integrality gaps that are challenging with the standard approaches.
\end{abstract}

\section{Introduction}

Given vectors $\bm{a},\bm{c}\in \R^n$, a \emph{Stieltjes} matrix $\bm{Q}\in \R^{n\times n}$ (that is, $\bm{Q}\succ 0$ and $Q_{ij}\leq 0$  for $i\neq j$), and a convex set $C\subseteq \R^{2n}$, consider the mixed-integer quadratic optimization (MIQO) problem 
\begin{subequations}\label{eq:opt}
\begin{align}
\min_{\bm{x}\in \R^n,\bm{z}\in \{0,1\}^n}\;&\bm{a^\top x}+\bm{c^\top z}+\bm{x^\top Q x}\\
\text{s.t.}\;
&\bm{x}\circ (\bm{e}-\bm{z})=\bm{0}\label{eq:opt_indicator}\\
&(\bm{x},\bm{z})\in C,\label{eq:opt_constr}\end{align}
\end{subequations} 
where $\bm{e}$ is an $n$-dimensional vector of ones. Problem \eqref{eq:opt} with a Stieltjes matrix $\bm{Q}$  arises naturally in statistical problems with graphical models, which we discuss in \S\ref{sec:applications}, and others.

A fundamental step towards solving \eqref{eq:opt} effectively is designing strong convex relaxations of the optimization problem. Several approaches have been proposed for MIQO in the literature, by exploiting low-dimensional cases \cite{frangioni2018decompositions,hga:2x2} or low-rank cases \cite{atamturk2019rank,wei2020convexification,wei2021ideal,shafieezadeh2023constrained,han2021compact,atamturk2020supermodularity}. In this paper, we turn our attention to a critical substructure given by 
$$X_Q\defeq\left\{(t,\bm{x},\bm{z})\in \R^{n+1}\times \{0,1\}^n: t\geq \bm{x^\top Q x},\;\eqref{eq:opt_indicator}\right\} \cdot$$
 Set $X_Q$ is the mixed-integer epigraph of a quadratic function with a Stieltjes matrix and indicators. \citet{atamturk2018strong} show that if $\bm{a}$ has all entries of the same sign, problem \eqref{eq:opt} with a Stieltjes matrix $\bm{Q}$ is polynomial-time solvable. However, a tractable convex relaxation of \eqref{eq:opt} that guarantees optimality under the same conditions has been missing. This paper presents such a convex relaxation. 

\subsection*{Outline} The rest of the paper is organized as follows. In \S\ref{sec:applications}, we discuss the applications of quadratic optimization with a Stieltjes matrix and indicators. In \S\ref{sec:preliminaries}, we provide the necessary background for the paper. In \S\ref{sec:convexification}, we convexify $X_Q$ by exploiting an underlying polyhedral structure. In \S\ref{sec:computation}, we present experimental results illustrating the computational impact of the proposed convexification. 

\subsection*{Notation} Given any $n\in \Z_+$, let $[n]\defeq\{1,\dots,n\}$. We denote vectors and matrices in \textbf{bold}. Moreover, $\bm{0}$ denotes either a vector or matrix of zeros (whose dimension is clear from the context), $\bm{e}$ and $\bm{E}$ denote a vector and matrix of ones, respectively. Given a set $S\subseteq [n]$ we let $\bm{e_S}\in \R^n$ denote the indicator vector of $S$, i.e., $(e_S)_i=1$ if $i\in S$ and $(e_S)_i=0$ otherwise. Moreover, given $i\in S$, we let $\bm{e}_i\defeq\bm{e}_{\{i\}}$ denote the $i$-th standard vector of $\R^n$. Similarly, we let $\bm{E_{ij}}$ be the square matrix whose dimensions can be inferred from the context, with a 1 in the $(i,j)$-th position and $0$ elsewhere. 

We let $\Sp^n$ denote the cone of positive semidefinite $n\times n$ matrices, and $\Spp^n$ the set of positive definite $n\times n$ matrices. Given a matrix $\bm{W}$, we let $\bm{W}^\dagger$ denote its pseudoinverse.  A special case of the pseudoinverse that is used throughout the paper pertains to matrices of the form $\bm{W}=\begin{pmatrix}\bm{A}&\bm{0}\\
\bm{0}&\bm{0}\end{pmatrix}$ where $\bm{A}$ is invertible, in which case $\bm{W}^\dagger=\begin{pmatrix}\bm{A^{-1}}&\bm{0}\\
\bm{0}&\bm{0}\end{pmatrix}$.  Given two matrices $\bm{W_1}$ and $\bm{W_2}$ of the same dimension, we let $\bm{W_1}\circ \bm{W_2}$ denote their Hadamard (entrywise) product. Given a matrix $\bm{W}\in \R^{n\times n}$ and set $S$, denote by $\bm{W_S}\in \R^{S\times S}$ the submatrix of $\bm{W}$ induced by $S$. Moreover, given $\bm{W}\in \R^{n\times n}$ and $\bm{S}\subseteq [n]$, observe that $\left(\bm{W}\circ \bm{e_Se_S^\top}\right)\in \R^{n\times n}$ is the matrix that coincides with $\bm{W_S}$ in the entries indexed by $S$ and is $0$ elsewhere. To simplify the notation, given $\bm{W}\in \R^{n\times n}$ and $S\subseteq [n]$, we let 
$$\pinv{W}{S}\defeq\left(\bm{W}\circ \bm{e_Se_S^\top}\right)^\dagger;$$ if $S$ corresponds to the first indices of $[n]$, then note that 
$\pinv{W}{S}=\begin{pmatrix}\bm{W_S^{-1}}&\bm{0}\\
\bm{0}&\bm{0}\end{pmatrix}\cdot$

\section{MIQOs with Stieltjes matrices}\label{sec:applications}

Problem \eqref{eq:opt} with a Stieltjes matrix $\bm{Q}$  arises naturally in statistical problems with graphical models, which we discuss in \S\ref{sec:direct}. Set $X_Q$ is critical to such problems because a convex relaxation of \eqref{eq:opt} can be obtained as 
\begin{align*}
\min_{(t,\bm{x},\bm{z})\in \R^{2n+1}}\;&\bm{a^\top x}+\bm{c^\top z}+t\\
\text{s.t.}\;
&(t,\bm{x},\bm{z})\in\conv(X_Q),\;(\bm{x},\bm{z})\in C,\end{align*}
and the relaxation is exact if $C=\R^{2n}$.

Understanding $\text{conv}(X_Q)$ for Stieltjes matrices can also be helpful in solving MIQO problems with non-Stieltjes quadratic matrices. In such cases, a standard approach in the literature is to decompose $\bm{Q}$ into simpler matrices of the form $\bm{Q}=\bm{Q_0}+\sum_{k\in K}\bm{Q_k}$ for some index set $K$, where $\bm{Q_k}$ are ``simple" matrices and $\bm{Q_0}\succeq 0$ is a remainder ``complicated" matrix. Then relaxations of problem \eqref{eq:opt} can be obtained as
\begin{subequations}\label{eq:decomp}
\begin{align}
\min_{(\bm{t},\bm{x},\bm{z})\in \R^{|K|+2n}}\;&\bm{a^\top x}+\bm{c^\top z}+\bm{x^\top Q_0 x}+\sum_{k\in K}t_k\\
\text{s.t.}\;
&(t_k,\bm{x},\bm{z})\in\conv(X_{Q_k})\qquad \forall k\in K\\
&(\bm{x},\bm{z})\in C.
\end{align}
\end{subequations}
The most prevalent relaxation in the literature to tackle \eqref{eq:opt} is the \emph{perspective relaxation} \cite{Frangioni2006,Frangioni2007,akturk2009strong,Gunluk2010}, which is a special case of the approach outlined where $|K|=1$ and $\bm{Q_1}$ is a diagonal positive definite matrix. A good understanding of relaxations of $X_Q$ allows one to extend the perspective relaxation to more general (non-diagonal) classes of Stieltjes matrices. In fact, as we discuss in \S\ref{sec:StieltjesEquiv}, matrices $\bm{Q_k}$ are not required to be Stieltjes to utilize the results of this paper.

\subsection{A direct application} \label{sec:direct}
Mixed-integer convex quadratic problems with Stieltjes matrices arise, for example, in inference problems with Besag-York-Molli\'e graphical models \cite{besag1991bayesian}. In this class of graphical models,  the vertex set $V$ of graph $\mathcal G=(V,E)$ represents latent random variables, $Y$, and the edge set $E$ represents relationships between random variables. The existence of an edge $[i,j]\in E$ indicates that the random variables $i$ and $j$ should have similar values. Thus, the values of the unobserved random variables can be estimated as the optimal solution of the optimization problem
\begin{subequations}\label{eq:BYM}
\begin{align}
\min_{\bm{x}\in \R^V,\bm{z}\in \{0,1\}^V}\;&\sum_{i\in V}a_i(y_i-x_i)^2+\sum_{[i,j]\in E}a_{ij}(x_i-x_j)^2\\
\text{s.t.}\;
&\eqref{eq:opt_indicator}-\eqref{eq:opt_constr},
\end{align}
\end{subequations}
where $y_i$ represents some noisy measurement of the value of random variable $Y_i$, $i\in V$, 
$\bm{a}\geq \bm{0}$ are parameters to be tuned, and the constraints incorporate logical constraints on the estimators. We refer the reader to \citet{han2022polynomial} for additional information on this class of combinatorial inference problems. 
As mentioned in the introduction, if $C=\mathbb{R}^n\times \{0,1\}^n$ or $C=\mathbb{R}_+^n\times \{0,1\}^n$, optimization problems with Stieltjes matrices and indicators can be solved in polynomial time. We refer the reader to  \citet{atamturk2018strong} for the case with $\bm{a}\leq \bm{0}$ and $C=\R_+\times \{0,1\}^n$.

\subsection{Stieljes-equivalent classes} \label{sec:StieltjesEquiv} In several cases, quadratic functions with non-Stieltjes matrices can be transformed into equivalent expressions with Stieltjes matrices. Specifically, given a vector $\bm{x}\in \R^n$, matrix $\bm{Q}\in \R^{n\times n}$ and index set $I\subseteq \{1,\dots,n\}$, define $$\bar x=\begin{cases}-x_i&\text{if }i\in I\\x_i&\text{otherwise,}\end{cases}\text{ and }\bar Q_{ij}=\begin{cases}-Q_{ij}&\text{if }\left|\{i,j\}\cap I\right|=1\\Q_{ij}&\text{otherwise.}\end{cases}$$ Then observe that $\bm{x^\top Qx}=\bm{\bar x^\top \bar Q\bar x}$. Thus, since $(t,\bm{x},\bm{z})\in X_Q\Leftrightarrow (t,\bm{\bar x},\bm{z})\in X_{\bar Q}$, we find that we can convexify sets with non-Stieltjes matrices $\bm{Q}$ provided that the transformed matrix $\bm{\bar Q}$ is Stieltjes. 

Improved formulations for convexifications using Stieltjes-equivalent classes have been presented in the literature. First, there have been notable efforts in studying $X_Q$ where $n=2$ \cite{atamturk2018sparse,han20232,frangioni2018decompositions,wei2023convex}. Naturally, any positive definite $2\times 2$ matrix is Stieltjes-equivalent: if $Q_{12}\leq 0$, then $\bm{Q}$ is already Stieljes; otherwise, after the substitution $\bar x_1=-x_1$, we recover an equivalent expression with a Stieltjes matrix. In another line of work, \citet{liu2021graph} describe the closure of the convex hull of $X_Q$ (in a polynomially-sized extended formulation) when $\bm{Q}$ is tridiagonal. It can also be verified that tridiagonal matrices are Stieljes equivalent.

\section{Preliminaries}\label{sec:preliminaries}
In this section, we discuss some preliminary results concerning the convexification of $X_Q$ relevant to the current paper. 
\begin{theorem}[\citet{wei2023convex}] \label{theo:convexification}Given any matrix $\bm{Q}\in \Spp^n$, the closure of the convex hull of $X_Q$ can be described in an extended formulation as 
\begin{align*}
\text{cl }\conv(X_Q)=\Big\{(t,\bm{x},\bm{z})\in \R^{2n+1}:&\exists \bm{W}\in \R^{n\times n}\text{ such that } \begin{pmatrix}\bm{W}&\bm{x}\\\bm{x^\top}&t\end{pmatrix}\in \Sp^{n+1},\\
&(\bm{z},\bm{W})\in \conv(P_Q)\Big\},
\end{align*}
where $P_Q\defeq \left\{(\bm{z},\bm{W})\in \{0,1\}^n\times \R^{n\times n}: \bm{W}=\left(\bm{Q}\circ \bm{zz^\top}\right)^{\dagger}\right\} \cdot$
\end{theorem}

Note that $P_Q$ is a finite set and, therefore, $\conv(P_Q)$ is a polytope. Consequently, we see from Theorem~\ref{theo:convexification} that convexification of $X_Q$ in a higher dimension reduces to the convexification of a polyhedral set, allowing for the use of theory and techniques from polyhedral theory. However, to utilize Theorem~\ref{theo:convexification} effectively, a major challenge must be overcome: characterizing or approximating set $\conv(P_Q)$, which has not been studied in the literature. Our goal in this paper is to close this gap for the special case of Stieltjes polytopes, as defined next.

\begin{definition}[Stieltjes polytope]\label{def:StieltjesPolytope} Given a Stieltjes matrix $\bm{Q}$, the Stieltjes polytope associated with $\bm{Q}$ is defined as
\begin{equation}\label{eq:defStieltjes}Z_Q=\conv\left(\left\{(\bm{e_S},\pinv{Q}{S})\right\}_{S\subseteq [n]}\right) \cdot\end{equation}
\end{definition}

\begin{example}\label{ex:3DExample1}
Let $n=3$, define $\bm{D}= \footnotesize\begin{pmatrix}3 & 0 & 0\\
0 & 4 & 0\\
0 &0 &3\end{pmatrix}$, $\bm{q}=\bm{e}$ and $\bm{Q}=\bm{D}-\bm{ee^\top}= \footnotesize\begin{pmatrix}2 & -1 & -1\\
-1 & 3 & -1\\
-1 & -1 & 2\end{pmatrix}.$ For the Stieltjes matrix $\bm{Q}$
the Stieltjes polytope $Z_{Q}$ is the convex hull of the following eight points in $\R^{12}$:
\footnotesize \begin{align*} 
&\left(\bm{0}, \begin{pmatrix}0 &0&0\\
0&0&0\\
0&0&0\end{pmatrix}\right),\left(\bm{e_1}, \begin{pmatrix}1/2 &0&0\\
0&0&0\\
0&0&0\end{pmatrix}\right),\left(\bm{e_2}, \begin{pmatrix}0 &0&0\\
0&1/3&0\\
0&0&0\end{pmatrix}\right),\\
&\left(\bm{e_3}, \begin{pmatrix}0 &0&0\\
0&0&0\\
0&0&1/2\end{pmatrix}\right),\left(\bm{e_{\{1,2\}}}, \begin{pmatrix}3/5 &1/5&0\\
1/5&2/5&0\\
0&0&0\end{pmatrix}\right),\left(\bm{e_{\{2,3\}}}, \begin{pmatrix}0 &0&0\\
0&2/5&1/5\\
0&1/5&3/5\end{pmatrix}\right),\\
&\left(\bm{e_{\{1,3\}}}, \begin{pmatrix}2/3 &0&1/3\\
0&0&0\\
1/3&0&2/3\end{pmatrix}\right),
\left(\bm{e_{\{1,2,3\}}}, \begin{pmatrix}5/3 &1&4/3\\
1&1&1\\
4/3&1&5/3\end{pmatrix}\right). 
\end{align*}\normalsize
\hfill $\blacksquare$
\end{example}

Since describing $Z_Q$ requires finding a polyhedral description of the inverses of submatrices of $\bm{Q}$, we review two technical lemmas concerning matrix inversion that will be used throughout the paper. 

\begin{lemma}[Blockwise inversion, \citet{lu2002inverses}]~\label{lem:blockwiseInversion} The inverse of
a non-singular square matrix $\bm{R}= \begin{pmatrix}\bm{A} & \bm{B}\\\bm{C}&\bm{D}\end{pmatrix}$ is given by
\small$$\bm{R^{-1}}= \begin{pmatrix}\bm{A^{-1}}+\bm{A^{-1}}\bm{B}(\bm{D}-\bm{CA^{-1}B})^{-1}\bm{CA^{-1}} & -\bm{A^{-1}B}(\bm{D}-\bm{CA^{-1}B})^{-1}\\(\bm{D}-\bm{CA^{-1}B})^{-1}\bm{CA^{-1}}&(\bm{D}-\bm{CA^{-1}B})^{-1}\end{pmatrix}.$$\normalsize
\end{lemma}

We apply Lemma~\ref{lem:blockwiseInversion} to a special class of matrices. We summarize the relevant results in the following corollary.
\begin{corollary}\label{cor:inversion} The inverse of
    a positive definite matrix $\bm{R}= \begin{pmatrix}\bm{A} & \bm{v}\\\bm{v^\top}&d\end{pmatrix}$, where $\bm{A}\in \R^{n\times n}$, $v\in \R^n$ and $d\in \R_+$ is given by \small$$\bm{R^{-1}}= \begin{pmatrix}\bm{A^{-1}}+\bm{A^{-1}}\bm{v}(d-\bm{v^\top A^{-1}v})^{-1}\bm{v^\top A^{-1}} & -\bm{A^{-1}v}(d-\bm{v^\top A^{-1}v})^{-1}\\(d-\bm{v^\top A^{-1}v})^{-1}\bm{v^\top A^{-1}}&(d-\bm{v^\top A^{-1}v})^{-1}\end{pmatrix} \cdot$$\normalsize
    Moreover, letting $\bm{u}\defeq\begin{pmatrix}-\bm{A^{-1}v}\\1\end{pmatrix}$, the identity 
\begin{equation}\label{eq:diffInverses}\begin{pmatrix}\bm{A} & \bm{v}\\\bm{v^\top}&d\end{pmatrix}^\dagger - \begin{pmatrix}\bm{A} & \bm{0}\\\bm{0^\top}&0\end{pmatrix}^\dagger=\frac{1}{d-\bm{v^\top A^{-1}v}}\bm{uu^\top}\end{equation}
holds.
\end{corollary}
Note that since the difference in \eqref{eq:diffInverses} is a rank-one matrix and, therefore, it is positive semi-definite. A recursive application of this property also wields that if $T\supseteq S$, then $\pinv{W}{T}\succeq \pinv{W}{S}$. The second lemma we introduce concerns the inverses of Stieltjes matrices in particular.

\begin{lemma}[Supermodular inverses, \citet{atamturk2018strong}] \label{lem:supermodular}Given a matrix $\bm{Q}\in \R^{n\times n}$ and a pair of indices $i,j\in [n]$, define the set function $\theta_{ij}(S)\defeq(\pinv{Q}{S}^\dagger)_{ij}$ as the $(i,j)$-th entry of matrix $\pinv{Q}{S}$. If $\bm{Q}$ is a Stieltjes matrix, then $\theta_{ij}$ is a non-decreasing supermodular function for all pairs of indices $i,j\in [n]$.  
\end{lemma}

Observe that since $\theta_{ij}(\emptyset)=0$ and $\theta_{ij}$ is non-decreasing, Lemma~\ref{lem:supermodular} implies that inverses of Stieltjes matrices are non-negative, a fact that is observed in \cite{plemmons1977m} (along with several other properties of Stieltjes matrices).

\section{Facial structure of Stieltjes polytopes}\label{sec:convexification}

In this section, we study convexifications of Stieltjes polytopes as defined in Definition~\ref{def:StieltjesPolytope}. Throughout, matrix $\bm{Q}\in \R^{n\times n}$ is assumed to be Stieltjes.
\begin{proposition}\label{prop:equalities} Any point $(\bm{z},\bm{W})\in Z_Q$ satisfies the following properties:
\begin{enumerate}
\item $W_{ij}=W_{ji}$ for all $1\leq i< j\leq n$,
\item $\sum_{j=1}^n Q_{ij}W_{ij}=z_i$ for all $i=1,\dots,n$,
\item $W_{ij}=0$ for all $i,j$ such that $Q_{ij}^{-1}=0$,
\item $W_{ij}\geq 0$ for all $i,j$.
\end{enumerate}
\end{proposition}
\begin{proof}
We first check the validity of the equalities. The first set of equalities follows since all extreme points of $Z_Q$ have symmetric matrices. The second set follows from \cite[Proposition 6]{wei2023convex}. For the third set, from Lemma~\ref{lem:supermodular} (non-negative and non-decreasing $\theta$) it follows that extreme points satisfy $0\leq W_{ij}\leq Q_{ij}^{-1}$: clearly, if $Q_{ij}^{-1}=0$, then $W_{ij}=0$ holds. Finally, the fourth set of inequalities follows directly from the non-negativity of inverses of Stieltjes matrices. 
\end{proof}

It is evident from Proposition~\ref{prop:equalities} that $Z_Q$ is not full-dimensional. In this paper, we study the relaxation induced by the upper-bound constraints
\begin{align*}
P_Q^\leq\defeq& \left\{(\bm{z},\bm{W})\in \{0,1\}^n\times \R^{n\times n}: \bm{W}\leq\left(\bm{Q}\circ \bm{zz^\top}\right)^{\dagger}\right\} \cdot
\end{align*}
Defining $Z_Q^\leq \defeq \conv(P_Q^\leq)$, it 
is easy to show that $Z_Q^\leq$ is full-dimensional.

Given $i,j\in [n]$ and $S\subseteq [n]$, let $\theta_{ij}(S)=(\pinv{Q}{S})_{ij}$ be the function introduced in Lemma~\ref{lem:supermodular}, and given $k\in [n]\setminus S$, let $$\rho_{ij}(k;S)\defeq \theta_{ij}(S\cup \{k\})-\theta_{ij}(S).$$ 
Finally, let $\bm{R}(k;S)$ be the matrix that collects functions $\rho$ in its entries; that is, $R(k;S)_{ij}=\rho_{ij}(k;S)$. Recall, from Corollary~\ref{cor:inversion} and the discussion immediately thereafter, that $\bm{R}$ is a rank-one matrix for all values of $k$ and $S$. 
Supermodularity of $\theta_{ij}$ immediately leads to a class of valid inequalities.

\subsection{Valid inequalities for $P_Q^\leq$} We now study the facial structure of $Z_Q^\leq$, which bounds matrix $\bm{W}$ from above. The facets of $Z_Q^\leq$ are given by the \emph{polymatroid inequalities} \cite{atamturk2008b}, and are a direct consequence of supermodularity of $\theta$. For any permutation $\bm{\pi}=(\pi_1,\pi_2,\dots,\pi_n)$ of $[n]$, define $S_0^\pi\defeq\emptyset$ and for $k\in [n]$, define set 
$$S_k^\pi\defeq\left\{\pi_1,\pi_2,\dots,\pi_k\right\} \cdot$$
\begin{proposition}
For any $i,j\in [n]$ and any permutation $\bm{\pi}=(\pi_1,\pi_2,\dots,\pi_n)$ of $[n]$, the inequality
\begin{align}
    W_{ij}&\leq \sum_{k=1}^n \rho_{ij}(\pi_k;S_{k-1}^\pi)z_{\pi_k}\label{eq:valid_ij}
\end{align}
    is valid for $Z_Q^{\leq}$.
\end{proposition}
\begin{proof}
Inequality \eqref{eq:valid_ij} is a polymatroid inequality, necessary to describe the Lov\'asz extension, which describes the concave envelope of a supermodular function \cite{lovasz1983submodular, AN:submodular-polarity}. 
\end{proof}
Observe that, given permutation $\bm{\pi}$, inequalities \eqref{eq:valid_ij} can be written compactly (for all $i,j\in [n]$) as the matrix inequalities
\begin{equation}\label{eq:valid_matrix}
    \bm{W}\leq \sum_{k=1}^n \bm{R}(\pi_k;S_{k-1}^\pi)z_{\pi_k}.
\end{equation}

\setcounter{example}{0}
\begin{example}[Continued]
For matrix $\bm{Q}= {\footnotesize \begin{pmatrix}2 & -1 & -1\\
-1 & 3 & -1\\
-1 & -1 & 2\end{pmatrix}}$ and permutation $\bm{\pi}= (1,2,3)$, inequalities \eqref{eq:valid_matrix} reduce to
{\footnotesize$$\begin{pmatrix}W_{11}&W_{12}&W_{13}\\
W_{21}&W_{22}&W_{23}\\
W_{31}&W_{32}&W_{33}\end{pmatrix}\leq \begin{pmatrix}1/2 &0&0\\
0&0&0\\
0&0&0\end{pmatrix}z_1+\begin{pmatrix}1/10 &1/5&0\\
1/5&2/5&0\\
0&0&0\end{pmatrix}z_2+\begin{pmatrix}16/15 &4/5&4/3\\
4/5&3/5&1\\
4/3&1&5/3\end{pmatrix}z_3. \blacksquare $$}
\end{example}

\subsection{Strength of the inequalities} We now show that inequalities \eqref{eq:valid_matrix} are indeed strong. 

\begin{proposition}\label{prop:facet}
Inequality \eqref{eq:valid_ij} is facet-defining for $\conv(P_Q^\leq)$.
\end{proposition}
\begin{proof}
In Table~\ref{tab:affinePointsPolymatroid}, we provide  $(n+n^2)$ affinely independent points in $P_Q^\leq$ such that \eqref{eq:valid_ij} holds at equality. 
\begin{table}[!h]
\begin{center}
\caption{Affinely independent points in $P_Q^\leq$ satisfying \eqref{eq:valid_ij} at equality.}
\label{tab:affinePointsPolymatroid}
\begin{tabular}{c c | c c}
\hline \hline
\textbf{\#}&\textbf{Indices}& $\bm{z}$ &$\bm{W}$\\
\hline
1&-&$\bm{e}$&$\bm{Q^{-1}}$\\
2&$\forall k,\ell\in [n]$ with $(k,\ell)\neq (i,j)$ and $Q^{-1}_{ij}\neq 0$&$\bm{e}$&$\bm{Q^{-1}}-\bm{E_{k\ell}}$\\
3&$\forall k\in [n]$&$\bm{e_{S_k^\pi}}$&$\pinv{Q}{S_k^\pi}$\\
\hline \hline
\end{tabular}
\end{center}
\end{table}

$\bullet$ Point \#1 belongs to $P_Q$ and, therefore, it also belongs to $P_Q^\leq$; if $\bm{z}=\bm{e}$, then inequality \eqref{eq:valid_ij} reduces to $W_{ij}\leq \theta_{ij}(\bm{e})=Q_{ij}^{-1}$, and thus the inequality is satisfied at equality at this point. 

$\bullet$ Points \#2  correspond to $n^2-1$ points which are obtained by subtracting a non-negative quantity from the feasible point \#1 and, therefore, also belong to $P_Q^\leq$; the inequality is tight for the same reason as point \#1. Points \#2 are affinely independent from previously introduced points because they are the first points with $W_{k\ell}\neq Q_{k\ell}^{-1}$.

$\bullet$ Points \#3  correspond to $n$ points belonging to $P_Q$; therefore, they also belong to $P_Q^\leq$; if $\bm{z}=\bm{e_{S_k^\pi}}$, then inequality \eqref{eq:valid_ij} reduces to $W_{ij}\leq \theta_{ij}(\bm{e_{S_k^\pi}})=\left(\pinv{Q}{S_k^\pi}\right)_{ij}$, and thus the inequality is satisfied at equality at this point. Finally, points \#3 are affinely independent from others because they are the first points with $z_k\neq 1$.
\end{proof}

From Proposition~\ref{prop:facet}, we see that all inequalities~\eqref{eq:valid_ij} for all combinations of $i,j\in [n]$ and all permutations $\bm{\pi}$ are necessary to describe $Z_Q^\leq$. We now show that they are, along with the bound inequalities, sufficient.

\begin{theorem}\label{theo:hull}
	Inequalities \eqref{eq:valid_matrix} (for all permutations $\bm{\pi}$) and bound constraints $\bm{0}\leq\bm{z}\leq \bm{e}$ describe $\conv(P_Q^\leq)$.
\end{theorem}
\begin{proof}
	We show that, for any $\bm{c}\in \R^n$ and $\bm{\Sigma}\in \R^{n\times n}$, the optimization problems 
	\begin{align}
		\min_{\bm{z},\bm{W}}&\;\bm{c^\top z}+\langle \bm{\Sigma},\bm{W}\rangle \text{ s.t. }(\bm{z},\bm{W})\in P_Q^{\leq}\label{eq:opt_exact}\\
			\min_{\bm{z},\bm{W}}&\;\bm{c^\top z}+\langle \bm{\Sigma},\bm{W}\rangle \text{ s.t. }\eqref{eq:valid_matrix},\; \bm{0}\leq \bm{z}\leq \bm{e}\label{eq:opt_relax}
	\end{align}
are equivalent; that is, either both are unbounded, or there exists an optimal solution of \eqref{eq:opt_relax} that is feasible for \eqref{eq:opt_exact}.

Note that if $\Sigma_{ij}> 0$ for any $i,j\in [n]$, then both problems are unbounded by letting $W_{ij}\to -\infty$. Therefore, we assume $\bm{\Sigma}\leq \bm{0}$. In this case, in optimal solutions of \eqref{eq:opt_exact}, $\bm{W}$ will be set to its upper bound, i.e., $\bm{W}=\left(\bm{Q}\circ \bm{zz^\top}\right)^{\dagger}$ holds. By abuse of notation, let $\theta_{ij}(\bm{z})\defeq \theta_{ij}(S_z)$, where $S_z=\left\{i\in [n]: z_i=1\right\}$; since $\left(\left(\bm{Q}\circ \bm{zz^\top}\right)^{\dagger}\right)_{ij}=\theta_{ij}(\bm{z})$, we find that \eqref{eq:opt_exact} is equivalent to 
\begin{equation}\label{eq:submodular-Min}\min_{\bm{z}\in \{0,1\}^n}\;\bm{c^\top z}+\sum_{i=1}^n\sum_{j=1}^n\Sigma_{ij}\theta_{ij}(\bm{z}).\end{equation} 
Recalling that functions $\theta_{ij}$ are supermodular (Lemma~\ref{lem:supermodular}) and $\Sigma_{ij}\leq 0$ for all $i,j$, we find that $\Theta(\bm{z})\defeq \sum_{i=1}^n\sum_{j=1}^n\Sigma_{ij}\theta_{ij}(\bm{z})$ is a submodular function. Therefore, it follows that \eqref{eq:submodular-Min} is equivalent to minimization over its Lov\'asz extension \cite{lovasz1983submodular}, which is precisely \eqref{eq:opt_relax}.
\end{proof}

Now consider the relaxation of \eqref{eq:opt} induced by Theorem~\ref{theo:convexification}, but using only inequalities \eqref{eq:valid_matrix} and bound constraints instead of the full description of $P_Q$:
\begin{subequations}\label{eq:sdpStieltjes}
	\begin{align}
		\min_{\bm{x},\bm{z},\bm{W},t}\;&\bm{a^\top x}+\bm{c^\top z}+t\\
		\text{s.t.}\;&\begin{pmatrix}\bm{W}&\bm{x}\\\bm{x^\top}&t\end{pmatrix}\in \Sp^{n+1}\label{eq:sdpStieltjes_psd}\\
		& \bm{W}\leq \sum_{k=1}^n \bm{R}(\pi_k;S_{k-1}^\pi)z_{\pi_k}\quad \text{for all permutations }\bm{\pi} \text{ of } [n]\label{eq:sdpStieltjes_polymatroid}\\
            &(\bm{x},\bm{z})\in C\\
		&\bm{x}\in \R^n,\; \bm{z}\in [0,1]^n,\; t\in \R_+,\;\bm{W}\in \R^{n\times n}. \label{eq:sdpStieltjes_bounds}
	\end{align}
\end{subequations}
Because constraints \eqref{eq:sdpStieltjes_polymatroid} are a relaxation of constraints $(\bm{z},\bm{W})\in P_Q$, we find from Theorem~\ref{theo:convexification} that \eqref{eq:sdpStieltjes} is indeed a valid relaxation. In general, \eqref{eq:sdpStieltjes} can be weak: in fact, it can be unbounded unless constraints $\bm{W}\geq \bm{0}$ are also added.  Nonetheless,
as we now show, under the specific conditions stated in \cite{atamturk2018strong} for polynomial-time solvability of \eqref{eq:opt}, the relaxation \eqref{eq:sdpStieltjes} is exact. Given a Stieltjes matrix $\bm{Q}$, define the optimization
\begin{align}\label{eq:optStieltjes}
\min_{\substack{\bm{x}\in \R^n,\bm{z}\in \{0,1\}^n\\ \bm{x}\circ (\bm{e}-\bm{z})=0}}\;\bm{a^\top x}+\bm{c^\top z}+\bm{x^\top Q x},
\end{align}
which is the special case of \eqref{eq:opt} with $C=\R^{2n}$. 

\begin{proposition}\label{prop:exactness}
    If $\bm{a}\leq \bm{0}$ or $\bm{a}\geq \bm{0}$, and $C=\R^{2n}$, then there exists an optimal solution of \eqref{eq:sdpStieltjes} that is also optimal for \eqref{eq:optStieltjes}, with the same objective value. 
\end{proposition}
\begin{proof}
    The start of the proof follows the steps in \cite[Theorem 1]{wei2023convex}, which we repeat for completeness. Constraint \eqref{eq:sdpStieltjes_psd} is equivalent to the system~\cite{albert1969conditions}
    $$\bm{W}\succeq 0,\; t\geq \bm{x^\top W^\dagger x},\text{ and }\bm{WW^\dagger x}=\bm{x}.$$
    Therefore, variable $t$ can be easily projected out since any optimal solution satisfies $t= \bm{x^\top W^\dagger x}$. We can restate problem \eqref{eq:sdpStieltjes} as 
    \begin{subequations}\label{eq:sdpStieltjes2}
    	\begin{align}
    		\min_{\bm{x},\bm{z},\bm{W}}\;&\bm{a^\top WW^\dagger x}+\bm{c^\top z}+\bm{x^\top W^\dagger x}\\
    		\text{s.t.}\;&\bm{WW^\dagger x}=\bm{x},\; \bm{W}\in \Sp^n\label{eq:sdpStieltjes2_pseudo}\\
    		&\eqref{eq:sdpStieltjes_polymatroid}-\eqref{eq:sdpStieltjes_bounds}.
    	\end{align}
    \end{subequations}
Note that we use the equality in \eqref{eq:sdpStieltjes2_pseudo} to rewrite a linear term in the objective. We now project out variables $\bm{x}$: The KKT conditions associated with the continuous variables are
\begin{align*}
	&\bm{WW^\dagger x}=\bm{x}\\
	&\bm{a^\top WW^\dagger}+2\bm{W^\dagger}x+\bm{\lambda^\top} \left(\bm{WW^\dagger}-\bm{I}\right)=\bm{0},
\end{align*} 
which are satisfied by setting $\bm{x^*}=-\frac{1}{2}\bm{Wa}$ and $\bm{\lambda^*}=0$. Substituting $\bm{x}$ with its optimal value, we find that problem \eqref{eq:sdpStieltjes2} further simplifies to 
  \begin{subequations}\label{eq:sdpStieltjes3}
	\begin{align}
		\min_{\bm{z},\bm{W}}\;&\langle -\frac{1}{4}\bm{aa^\top},\bm{W}\rangle+\bm{c^\top z}\\
		\text{s.t.}\;
		&\eqref{eq:sdpStieltjes_polymatroid}, \bm{W}\in \Sp^n, \bm{0}\leq\bm{z}\leq\bm{e}.
	\end{align}
\end{subequations}
In the last step of the proof (which does not follow from \cite{wei2023convex}), we study the relaxation of \eqref{eq:sdpStieltjes3} obtained by removing constraint $\bm{W}\in \Sp^n$. In particular, we show that there exist optimal solutions $(\bm{\bar z},\bm{\bar W}, \bar t)$ of the relaxation such that: \textit{(i)} $\bm{\bar z}$ is integral; \textit{(ii)} $\bm{\bar W}$ is positive semidefinite; \textit{(iii)} $\bar t=-\frac{1}{4}\bm{a_S^\top Q_S^{-1}a_S}$, where $S=\{i\in [n]: \bar z_i=1\}$. As a consequence, $(\bm{\bar z},\bm{\bar W}, \bar t)$ is also optimal for \eqref{eq:sdpStieltjes3} and \eqref{eq:optStieltjes}, concluding the proof.

\pagebreak

Observe that $-\frac{1}{4}\bm{aa^\top}\leq \bm{0}$ due to the assumption that $\bm{a}$ is of the same sign. Therefore, if $\bm{W}\in \Sp^n$ is removed, then in optimal solutions of the relaxation, $\bm{\bar W}$ is equal to its upper bound, i.e., 
\begin{equation}\label{eq:WPermut}\bm{\bar W}=\min_{\pi\in \Pi}\sum_{k=1}^n \bm{R}(\pi_k;S_{k-1}^\pi)z_{\pi_k},\end{equation}  where $\Pi$ is the set of all permutations of $[n]$. In other words, the relaxation is equivalent to 
\begin{align}	\label{eq:Lovasz}	     \min_{\bm{0}\leq\bm{z}\leq \bm{1}}\;& \min_{\pi\in \Pi}\left\{\sum_{k=1}^n\left(\langle -\frac{1}{4}\bm{aa^\top}, \bm{R}(\pi_k;S_{k-1}^\pi)\rangle\right) z_{\pi_k}\right\}+\bm{c^\top z}.
\end{align}
Recognizing \eqref{eq:Lovasz} as a linear optimization problem over the Lov\'asz extension of a submodular function, we conclude that $\bm{\bar z}\in \{0,1\}^n$ in optimal solutions, proving \textit{(i)}. Since optimal permutations for \eqref{eq:WPermut}  
 correspond to nondecreasing orders of $\bm{\bar z}$ \cite{Edmonds1970}, letting $\tau=\|\bm{\bar z}\|_0$ we conclude that \begin{align*}\bm{\bar W}=\sum_{k=1}^{\tau} \bm{R}(\pi_k;S_{k-1})=\sum_{k=1}^\tau \left(\pinv{W}{S_k}-\pinv{W}{S_{k-1}}\right)=\pinv{W}{S_\tau},\end{align*}
where we defined $\pinv{W}{S_{0}}=\bm{0}$. In particular, $\bm{\bar W}\in \Sp^n$, proving \textit{(ii)}, and substituting $\bm{W}$ with its optimal value in \eqref{eq:Lovasz} we can prove \textit{(iii)}.
\end{proof}

\subsection{Separation algorithm}

We now discuss the separation problem: given any fixed point $(\bm{\bar z},\bm{\bar W})\in \R^{n}\times \R^{n\times n}$, how to find a violated inequality~\eqref{eq:valid_ij} if there exists one. Since inequality~\eqref{eq:valid_ij} is a polymatroid inequality, it follows that the most violated inequalities correspond to permutations $(\pi_1,\pi_2,\dots,\pi_n)$ such that $\bar z_{\pi_1}\geq \bar z_{\pi_2}\geq \dots\geq \bar z_{\pi_n}$ \cite{Edmonds1970}. In particular, since the permutation depends only on the values of $\bm{\bar z}$, we find that most violated permutations coincide for all values of indices $i,j$ in \eqref{eq:valid_ij}. Therefore, using the matrix notation \eqref{eq:valid_matrix}, we see that a straightforward way to compute a violated inequality is simply to compute matrices $\pinv{Q}{S_k^\pi}$ for all $k\in [n]$, which requires inverting $\mathcal{O}(n)$ matrices. A faster approach to compute the coefficients of inequalities \eqref{eq:valid_matrix} consists of using a Cholesky decomposition.

Indeed, consider the properties of coefficient matrices $\{\bm{R}(\pi_k;S_{k-1})\}_{k=1}^n$ in inequalities \eqref{eq:valid_matrix}. First, they are rank-one matrices; that is, there exists $\bm{v_k}\in \R^n$ such that $\bm{R}(\pi_k;S_{k-1})=\bm{v_kv_k^\top}$. Second, they add up to $\bm{Q}^{-1}$; that is, $\sum_{k=1}^n\bm{R}(\pi_k;S_{k-1})=\sum_{k=1}^n\bm{v_kv_k^\top}=\bm{Q^{-1}}$. Third, the entries corresponding to indices that have not appeared yet in the permutation vanish; that is, $R(\pi_k;S_{k-1})_{ij}=0$ if $\max\{i,j\}>k$, or equivalently $(v_k)_i=0$ if $i>k$. If we define a matrix $\bm{V}\in \R^{n\times n}$ such that its $i$-th column is precisely $\bm{v_i}$, these properties are equivalent to $\bm{VV^\top}=\bm{Q^{-1}}$ and $\bm{V}$ is upper triangular (where elements are ordered according to the permutation $\pi$). Note that these properties are similar to the ones corresponding to a Cholesky decomposition of $\bm{Q^{-1}}$, except that the matrix in the Cholesky decomposition is lower triangular instead of upper triangular. To account for this difference, it suffices to compute the Cholesky decomposition in the \emph{reverse order} of the permutation, and because $\bm{Q^{-1}}\in \Spp^n$, the Cholesky decomposition is the unique matrix satisfying the required properties, and thus indeed coincides with the coefficients in inequalities \eqref{eq:valid_matrix}. We summarize the separation algorithm in Proposition~\ref{prop:separation} below.

\begin{proposition}\label{prop:separation}
Algorithm \ref{alg:separation} produces up to $\mathcal{O}(n^2)$ violated inequalities, 
one for each combination of elements $i,j\in [n]$, if there exists any.
\end{proposition}

\begin{algorithm}[h]
	\caption{Separation procedure}	
	\label{alg:separation} 
	\begin{algorithmic}[1]
	\renewcommand{\algorithmicrequire}{\textbf{Input:}}
	\renewcommand{\algorithmicensure}{\textbf{Output:}}
	\Require Point $\bm{\bar z}\in \R^n$.
	\Ensure Most violated inequalities.
		\State Find a permutation $\bm{\pi}$ satisfying $\bar z_{\pi_1}\leq \bar z_{\pi_2}\leq \dots\leq \bar z_{\pi_n}$ \Comment{Sorting}\label{line:sorting}
 \State Compute Cholesky decomposition $\bm{Q^{-1}}=\bm{VV^\top}$ according to order $\bm{\pi}$\label{line:cholesky}
 \State \Return inequalities
 $$\bm{W}\leq \sum_{k=1}^n \left(\bm{v_{\pi_k}v_{\pi_k}^\top}\right)z_{\pi_k}$$
 where $\bm{v_{\pi_i}}$ denotes the $i$-th column of $\bm{V}$.
	\end{algorithmic}
\end{algorithm}

We emphasize that the order in line~\ref{line:sorting} of Algorithm~\ref{alg:separation} is non-decreasing, instead of the natural non-increasing order that arises often with polymatroid inequalities, since we use the \emph{reverse order} of the permutation in computing the Cholesky decomposition as we discussed in the preceding paragraph.

We now discuss the runtime of Algorithm~\ref{alg:separation}.
Since sorting the variables (line~\ref{line:sorting}) can be done in $\mathcal{O}(n\log n)$, we find that the complexity of Algorithm~\ref{alg:separation} is dominated by the cost of computing a Cholesky decomposition (line~\ref{line:cholesky}), which is usually $\mathcal{O}(n^3)$. Note that line~\ref{line:cholesky} also requires inverting matrix $\bm{Q}$, but this operation (with cubic runtime as well) needs to be performed only once as preprocessing. In some cases, the runtime for a general Stieltjes matrix can be improved. For example, Lemma~\ref{cor:inversion} can be called recursively to compute the coefficients, requiring $\mathcal{O}(n)$ calls to a routine for matrix-vector product and matrix-matrix subtraction: if $\bm{Q}$ is sparse, then the vectors appearing in the computations are sparse as well, potentially improving runtimes for the multiplications (with appropriate data structures). Finally, we point out that the output of Algorithm~\ref{alg:separation} consists of $\mathcal{O}(n^2)$ inequalities with up to $n$ nonzeros per inequalities, thus formulating the inequalities in a solver already requires processing a cubic number of nonzeros, thus it is not possible to improve this runtime (at least with an off-the-shelf solver).

\pagebreak

\section{Computational results}\label{sec:computation}

In this section, we describe computational experiments performed to test the effectiveness of the proposed convexification. For the computational study, we consider problems of the form \eqref{eq:BYM}, arising from the inference of sparse Besag-York-Molli\'e graphical models as discussed in \S\ref{sec:direct}. Specifically, given a graph $\mathcal{G}=(V,E)$, we consider problem 
\begin{subequations}\label{eq:MRF}
\begin{align}
\min_{\bm{x}\in \R^{V},\bm{z}\in \{0,1\}^{V}}\;&\frac{1}{\sigma^2}\sum_{i\in V}(y_i-x_i)^2+\sum_{[i,j]\in E}(x_i-x_j)^2+\mu \sum_{i\in V}z_i\label{eq:MRF_obj}\\
\text{s.t.}\;
&\sum_{i\in V}z_i\leq k,\; \bm{x}\circ (\bm{e}-\bm{z})=\bm{0},
\end{align}
\end{subequations}
where $\sigma,\mu,k\in \R_+$ are given parameters. In our experiments, we solve \emph{cardinality-penalized problems} with $\mu>0$ and $k=|V|$ (modeling situations where the density of the model is penalized) as well as  \emph{cardinality-constrained problems} with $\mu=0$ and $k<|V|$ (modeling situations where the sparsity is directly specified). In either case, we solve \eqref{eq:MRF} for different values of $\sigma$, a parameter that is connected to the noise of the underlying model.

\begin{figure}[!h]
	\begin{center}
		\includegraphics[width=0.45\linewidth, trim={4cm 4cm 12cm 1cm},clip]{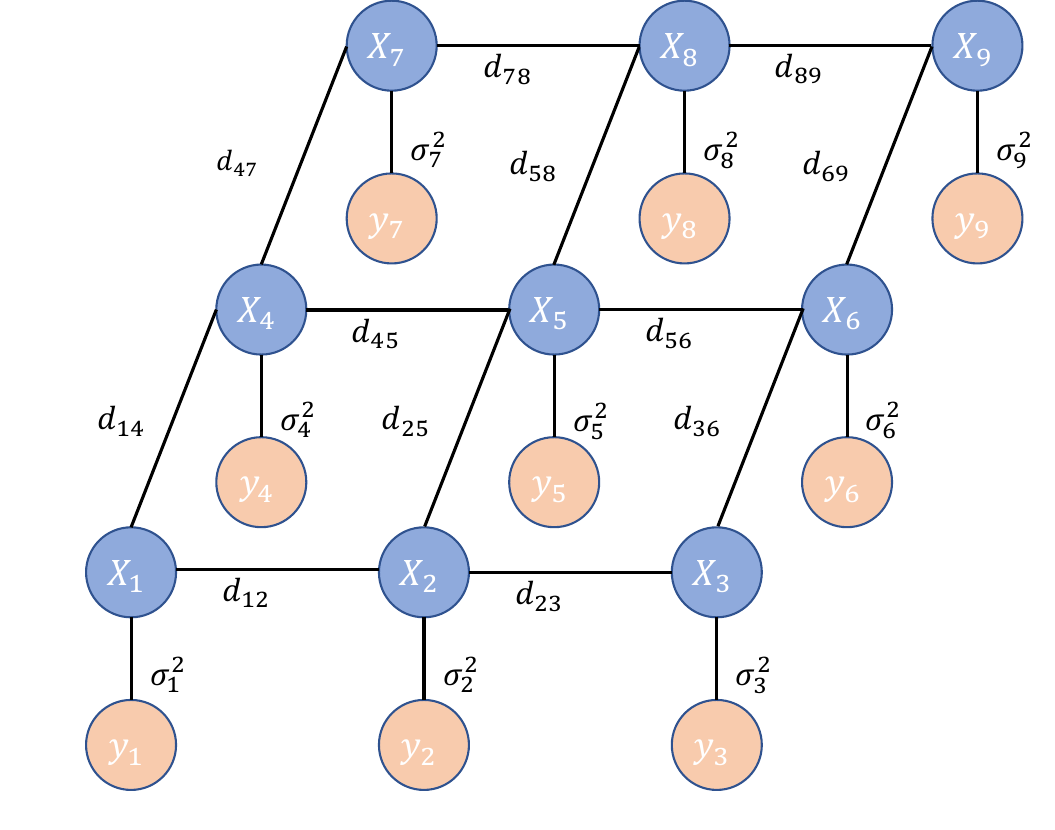}
		\caption{Grid graph for modeling the Besag-York-Molli\'e process; from \cite{he2021comparing}.}	
		\label{fig: 2D}
	\end{center}
\end{figure}

The data are generated following  \cite{he2021comparing} and are available online at \url{https://sites.google.com/usc.edu/gomez/data}.
Figure~\ref{fig: 2D} depicts the graph $\mathcal{G}$ used to model spatial processes. Graph $\mathcal{G}$ is given by a two-dimensional lattice; that is, vertices $V=[n]$ are arranged in a grid, with edges between horizontally and vertically adjacent vertices. We consider instances with grid sizes $10\times 10$, thus resulting in instances with $|V| = 100$. To generate the data, we create a sparse ``true" signal $\{X_i\}_{i\in V}$ with three non-negative spikes, each spike affecting a $3\times 3$ block of elements in $V$. Thus, the generated signal can have at most 27 non-zeros (the construction of the underlying true signal is explained in detail in Appendix~\ref{sec:construction}). We then generate noisy observations as $y_i=|X_i+\epsilon_i|$, where $\epsilon_i$ are i.i.d. samples from a Gaussian distribution $\mathcal{N}(0,\sigma^2)$. Note that since $\bm{y}\geq \bm{0}$, linear coefficients of variables $\bm{x}$ obtained by expanding the quadratic terms in \eqref{eq:MRF_obj} are non-positive.

In \S~\ref{sec:compute_method}, we describe the models tested, and in \S~\ref{sec:compute_poly}, we report the computational experiments.

\subsection{Models}\label{sec:compute_method}
We compare three reformulations or relaxations of the problem ~\eqref{eq:opt}. The former two are commonly used perspective formulations \cite{Frangioni2006,Gunluk2010}, while the latter is based on the valid inequalities proposed in \S\ref{sec:convexification}. We describe these formulations next. \\

\noindent
\texttt{pers-c}: Perspective relaxation with continuous variable $\bm{z}$ and big-M constraints:
\begin{subequations}\label{eq:MRFPersp}
\begin{align}
\min_{\bm{x},\bm{z}}\;&\frac{1}{\sigma^2}\|\bm{y}\|_2^2+\frac{1}{\sigma^2}\sum_{i\in V}\left(\frac{x_i^2}{z_i}-2y_ix_i\right)+\sum_{[i,j]\in E}(x_i-x_j)^2+\mu \sum_{i\in V}z_i\\
\text{s.t.}\;
&\sum_{i\in V}z_i\leq k,\; -M\bm{z}\leq \bm{x}\leq M\bm{z} \\
& \bm{x}\in \R^{V},\bm{z}\in [0,1]^{V},
\end{align}
\end{subequations}
where the value $M$ is large enough so that the big-M constraints are redundant when $z_i=1$. In our computations, the bounds are set to $M=10$. 
\\

\noindent
\texttt{pers-b}: Perspective reformulation \texttt{pers-c} with binary variables $\bm{z}\in \{0,1\}^{V}$. Using this formulation, problems are solved to optimality with branch-and-bound, closing the gap from the perspective relaxation.
\\

\noindent
\texttt{poly}: Polymatroid relaxation with the equalities and lower bounds given in Proposition \ref{prop:equalities} and the polymatroid inequalities as described in  ~\eqref{eq:sdpStieltjes}:
\begin{subequations}\label{eq:sdp-p_formulation}
	\begin{align}
		\min_{\bm{x},\bm{z},\bm{W},t}\;&\frac{1}{\sigma^2}\|\bm{y}\|_2^2-2\frac{1}{\sigma^2}\sum_{i\in V}y_ix_i+\mu\sum_{i\in V}z_i+t\label{eq:sdp-p-obj}\\		\text{s.t.}\;&\sum_{i\in V}z_i\leq k\\ &\begin{pmatrix}\bm{W}&\bm{x}\\\bm{x^\top}&t\end{pmatrix}\succeq 0\label{eq:sdp-p-first}\\
  & W_{ij}=W_{ji}\quad \text{for all } i,j\in V\\
 & \sum_{j\in V} Q_{ij}W_{ij}=z_i\quad \text{for all }i\in V \label{eq:sdp-p-third}\\
		& \bm{W}\leq \sum_{k\in V} \bm{R}(\pi_k;S_{k-1}^\pi)z_{\pi_k}\quad \text{for all permutations }\bm{\pi} \text{ of } V \label{eq:sdp-p-polymatroid}\\
		&\bm{x}\in \R^{V},\; \bm{z}\in [0,1]^{V},\; t\in \R_+,\;\bm{W}\in \R_+^{V\times V}, \label{eq:sdp-p-last}
	\end{align}
\end{subequations}
where $\bm{Q}$ is the Stieltjes matrix obtained by collecting all nonlinear terms in the objective and inequalities \eqref{eq:sdp-p-polymatroid} are those required to describe $\conv(P_Q^\leq)$. Note that \eqref{eq:sdp-p_formulation} is a reformulation of \eqref{eq:opt} under the conditions of Proposition~\ref{prop:exactness}, that is, in the cardinality-penalized instances, and is a relaxation otherwise.

\begin{remark}[Implementation of \texttt{poly}]\label{remark:poly-large} 
As the total number of polymatroid inequalities is $|V|!$, listing all of them in \eqref{eq:sdp-p_formulation} is impractical. Instead, we add \eqref{eq:sdp-p-polymatroid} as cutting planes; that is, we solve the \texttt{poly} problem ~\eqref{eq:sdp-p_formulation} by adding polymatroid inequalities \eqref{eq:sdp-p-polymatroid} as cutting planes iteratively using the separation procedure ~\ref{alg:separation} and stopping when the difference of the objective values between two subsequent iterations is small enough ($<10^{-3}$). 
\end{remark}

\subsection{Results}\label{sec:compute_poly}

The perspective formulations (\texttt{pers-c}, \texttt{pers-b}) are solved with Gurobi version 9.0.2 using a single thread. The time limit for the computations is set to one hour; all other configurations are set to default values. 
The polymatroid relaxation (\texttt{poly}) is solved with Mosek version 9.3 with the default settings. All experiments are run on a Lenovo laptop with a 1.9 GHz Intel$\circledR$Core$^{\text{TM}}$ i7-8650U CPU and 16 GB main memory.  For cardinality-constrained instances, we set $k=20$, and for cardinality-penalized instances, we choose $\mu$ so that the number of non-zeros of the estimator approximately matches the number of non-zeros of the underlying signal.

Tables~\ref{tab:negative-polymatroid-penalty} and \ref{tab:negative-polymatroid-cardinal} present results for values of $\sigma^2\in \{0.5,1.0,2.0,5.0,10.0\}$ and cardinality-penalized and cardinality-constrained instances, respectively. Each row represents the average over five instances generated with identical parameters.
The gap corresponding to each model compares the lower bound produced by the model with the best upper bound found by Gurobi with model \texttt{pers-b}. For model \texttt{pers-b}, we also report, under the \#Opt column, the number of instances that can be solved to optimality before the time limit, and, under the Gap column, we report the average optimality gap at termination. The number of iterations for model \texttt{poly} is the number of cutting plane rounds required before termination, with each iteration resulting in the addition of up to ${|V| \choose 2 }$ cuts. All times reported are in seconds.

\begin{table}[htbp]
  \centering
    \caption{Experiments with cardinality-penalized instances.}
    \begin{tabular}{c|c|cc|ccc|ccc}\hline \hline
          &       & \multicolumn{2}{c|}{\texttt{pers-c}} & \multicolumn{3}{c|}{\texttt{pers-b}} & \multicolumn{3}{c}{\texttt{poly}} \\\hline
    $\sigma^2$ & $\mu$    & Time  & Gap & \#Opt  & Time  & Gap   & \#Iter & Time  & Gap \\\hline 
    0.5   & 0.25  & 0.1  & 2.4\% & 5 & 2.1   & 0.0\%  & 4     & 102.4 & $3\cdot 10^{-8}$ \\
    1.0     & 0.12  & 0.1  & 5.7\% &1  & 2,998.3 & 1.7\%&  6     & 235.4 & $7\cdot 10^{-4}$ \\
    2.0     & 0.12  & 0.1  & 5.1\% &2  & 2,937.1 & 0.5\% & 6     & 218.3 & $2\cdot 10^{-7}$\\
    5.0     & 0.12  & 0.1  & 5.1\% &1  & 3,021.9 & 0.9\% & 6     & 235.6 & $3\cdot 10^{-6}$ \\
    10.0    & 0.12  & 0.1  & 5.1\% &2  & 2,990.6 & 0.8\% & 6     & 219.0 & $9\cdot 10^{-8}$ \\\hline \hline
    \end{tabular}
  \label{tab:negative-polymatroid-penalty}
\end{table}

\begin{table}[htbp]
  \begin{center}
    \caption{Experiments with cardinality-constrained instances.}
    \begin{tabular}{c|cc|ccc|ccc}\hline \hline
          & \multicolumn{2}{c|}{\texttt{poly}} & \multicolumn{3}{c}{\texttt{pers-b}} & \multicolumn{3}{|c}{\texttt{poly}} \\\hline
    $\sigma^2$ & Time  & Gap & \#Opt  & Time  & Gap   & \#Iter & Time  & Gap \\\hline
    0.5   & 0.1  & 3.3\% & 5 & 13.1  & 0.0\%   & 4     & 181.3 & $4\cdot 10^{-7}$ \\
    1.0     & 0.1  & 7.4\% & 0 & 3,600.0 & 1.5\% & 7     & 285.3 & $9\cdot 10^{-4}$ \\
    2.0     & 0.1  & 7.1\% & 1 & 3,017.3 & 1.3\% & 7     & 277.5 & $3\cdot 10^{-3}$ \\
    5.0     & 0.1  & 7.1\% & 1 & 2,990.3 & 1.4\% & 7     & 281.7 & $2\cdot 10^{-3}$ \\
    10.0    & 0.1  & 7.1\% & 1 & 2,988.9 & 1.3\% & 7     & 260.0 & $3\cdot 10^{-3}$ \\\hline \hline
    \end{tabular}
  \label{tab:negative-polymatroid-cardinal}
  \end{center}
\end{table}

\pagebreak

In both cases, we observe that the relaxations from the perspective relaxation can be solved in a fraction of a second and yield gaps between 2\% and 7\%: the gaps are larger for cardinality-constrained instances and for problems with larger values of $\sigma$. Indeed, for a larger value of $\sigma$, terms $x_i^2/z_i$ (crucial to the strength of the perspective relaxation) in the objective of \eqref{eq:MRFPersp} represent a relatively smaller portion of the objective, leading to larger gaps. When used in a branch-and-bound algorithm, it is possible to efficiently prove optimality if $\sigma^2=0.5$, but most instances with larger values of $\sigma$ cannot be solved within the time limit of one hour. 

For cardinality-penalized instances, model \texttt{poly} delivers optimal solutions on average in about 200 seconds. The runtimes are larger than those required to solve the perspective relaxation   \texttt{pers-c}  due to expenses associated with solving a semidefinite program with a large number of linear inequalities via cutting planes. Nonetheless, for the more challenging instances with $\sigma^2\geq 1$, the runtimes of \texttt{poly} are at least an order of magnitude less than those arising from the branch-and-bound method to solve \texttt{pers-b}; for these instances, both \texttt{poly} and \texttt{pers-b} are exact models. The runtimes of \texttt{poly} for cardinality-constrained instances are similar. Although in this case, due to the additional cardinality constraint, \texttt{poly} is not guaranteed to deliver exact solutions to problem \eqref{eq:MRF}, the gaps proved are almost $0\%$ and much smaller than those obtained after one hour of branching with  \texttt{pers-b}.

\section*{Acknowledgments}

Alper Atamt\"urk is supported, in part, by NSF AI Institute for Advances in Optimization Award 2112533 and the Office of Naval Research grant N00014-24-1-2149. Andr\'es G\'omez is supported, in part, by grant FA9550-22-1-0369 from the Air Force Office of Scientific Research. Simge K\"u\c{c}\"ukyavuz is supported, in part, by grant \#2007814 from the National Science Foundation.

\pagebreak

\bibliographystyle{abbrvnat}

\begin{thebibliography}{27}
\providecommand{\natexlab}[1]{#1}
\providecommand{\url}[1]{\texttt{#1}}
\expandafter\ifx\csname urlstyle\endcsname\relax
  \providecommand{\doi}[1]{doi: #1}\else
  \providecommand{\doi}{doi: \begingroup \urlstyle{rm}\Url}\fi

\bibitem[Akt{\"u}rk et~al.(2009)Akt{\"u}rk, Atamt{\"u}rk, and G{\"u}rel]{akturk2009strong}
M.~S. Akt{\"u}rk, A.~Atamt{\"u}rk, and S.~G{\"u}rel.
\newblock A strong conic quadratic reformulation for machine-job assignment with controllable processing times.
\newblock \emph{Operations Research Letters}, 37:\penalty0 187--191, 2009.

\bibitem[Albert(1969)]{albert1969conditions}
A.~Albert.
\newblock Conditions for positive and nonnegative definiteness in terms of pseudoinverses.
\newblock \emph{SIAM Journal on Applied Mathematics}, 17\penalty0 (2):\penalty0 434--440, 1969.

\bibitem[Atamt{\"u}rk and G{\'o}mez(2018)]{atamturk2018strong}
A.~Atamt{\"u}rk and A.~G{\'o}mez.
\newblock Strong formulations for quadratic optimization with {M}-matrices and indicator variables.
\newblock \emph{Mathematical Programming}, 170:\penalty0 141--176, 2018.

\bibitem[Atamt\"urk and G\'omez(2019)]{atamturk2019rank}
A.~Atamt\"urk and A.~G\'omez.
\newblock Rank-one convexification for sparse regression.
\newblock \emph{arXiv preprint arXiv:1901.10334}, 2019.

\bibitem[Atamt\"urk and G\'omez(2023)]{atamturk2020supermodularity}
A.~Atamt\"urk and A.~G\'omez.
\newblock Supermodularity and valid inequalities for quadratic optimization with indicators.
\newblock \emph{Mathematical Programming}, 201:\penalty0 295--338, 2023.

\bibitem[Atamt{\"u}rk and Narayanan(2008)]{atamturk2008b}
A.~Atamt{\"u}rk and V.~Narayanan.
\newblock Polymatroids and mean-risk minimization in discrete optimization.
\newblock \emph{Operations Research Letters}, 36:\penalty0 618--622, 2008.

\bibitem[Atamt{\"u}rk and Narayanan(2022)]{AN:submodular-polarity}
A.~Atamt{\"u}rk and V.~Narayanan.
\newblock Submodular function minimization and polarity.
\newblock \emph{Mathematical Programming}, 196:\penalty0 57--67, 2022.

\bibitem[Atamt\"urk et~al.(2021)Atamt\"urk, G\'omez, and Han]{atamturk2018sparse}
A.~Atamt\"urk, A.~G\'omez, and S.~Han.
\newblock Sparse and smooth signal estimation: Convexification of $\ell_0$ formulations.
\newblock \emph{Journal of Machine Learning Research}, 22:\penalty0 1--43, 2021.

\bibitem[Besag et~al.(1991)Besag, York, and Molli{\'e}]{besag1991bayesian}
J.~Besag, J.~York, and A.~Molli{\'e}.
\newblock Bayesian image restoration, with two applications in spatial statistics.
\newblock \emph{Annals of the Institute of Statistical Mathematics}, 43\penalty0 (1):\penalty0 1--20, 1991.

\bibitem[Edmonds(1970)]{Edmonds1970}
J.~Edmonds.
\newblock Submodular functions, matroids, and certain polyhedra.
\newblock In R.~Guy, H.~Hanani, N.~Sauer, and J.~Sch\"onenheim, editors, \emph{Combinatorial Structures and Their Applications}, pages 69--87. Gordon and Breach, 1970.

\bibitem[Frangioni and Gentile(2006)]{Frangioni2006}
A.~Frangioni and C.~Gentile.
\newblock Perspective cuts for a class of convex 0--1 mixed integer programs.
\newblock \emph{Mathematical Programming}, 106:\penalty0 225--236, 2006.

\bibitem[Frangioni and Gentile(2007)]{Frangioni2007}
A.~Frangioni and C.~Gentile.
\newblock {SDP} diagonalizations and perspective cuts for a class of nonseparable {MIQP}.
\newblock \emph{Operations Research Letters}, 35:\penalty0 181--185, 2007.

\bibitem[Frangioni et~al.(2020)Frangioni, Gentile, and Hungerford]{frangioni2018decompositions}
A.~Frangioni, C.~Gentile, and J.~Hungerford.
\newblock Decompositions of semidefinite matrices and the perspective reformulation of nonseparable quadratic programs.
\newblock \emph{Mathematics of Operations Research}, 45\penalty0 (1):\penalty0 15--33, 2020.

\bibitem[G{\"u}nl{\"u}k and Linderoth(2010)]{Gunluk2010}
O.~G{\"u}nl{\"u}k and J.~Linderoth.
\newblock Perspective reformulations of mixed integer nonlinear programs with indicator variables.
\newblock \emph{Mathematical Programming}, 124:\penalty0 183--205, 2010.

\bibitem[Han and G{\'o}mez(2021)]{han2021compact}
S.~Han and A.~G{\'o}mez.
\newblock Compact extended formulations for low-rank functions with indicator variables.
\newblock \emph{arXiv preprint arXiv:2110.14884}, 2021.

\bibitem[Han et~al.(2020)Han, G{\'o}mez, and Atamt{\"u}rk]{hga:2x2}
S.~Han, A.~G{\'o}mez, and A.~Atamt{\"u}rk.
\newblock 2x2 convexifications for convex quadratic optimization with indicator variables.
\newblock \emph{arXiv preprint arXiv:2004.07448}, 2020.

\bibitem[Han et~al.(2022)Han, G{\'o}mez, and Pang]{han2022polynomial}
S.~Han, A.~G{\'o}mez, and J.-S. Pang.
\newblock On polynomial-time solvability of combinatorial {Markov} random fields.
\newblock \emph{arXiv preprint arXiv:2209.13161}, 2022.

\bibitem[Han et~al.(2023)Han, G{\'o}mez, and Atamt{\"u}rk]{han20232}
S.~Han, A.~G{\'o}mez, and A.~Atamt{\"u}rk.
\newblock 2$\times$ 2-convexifications for convex quadratic optimization with indicator variables.
\newblock \emph{Mathematical Programming}, 202:\penalty0 95--134, 2023.

\bibitem[He et~al.(2024)He, Han, G{\'o}mez, Cui, and Pang]{he2021comparing}
Z.~He, S.~Han, A.~G{\'o}mez, Y.~Cui, and J.-S. Pang.
\newblock Comparing solution paths of sparse quadratic minimization with a {Stieltjes} matrix.
\newblock \emph{Mathematical Programming}, 204:\penalty0 517--566, 2024.

\bibitem[Liu et~al.(2023)Liu, Fattahi, G{\'o}mez, and K{\"u}{\c{c}}{\"u}kyavuz]{liu2021graph}
P.~Liu, S.~Fattahi, A.~G{\'o}mez, and S.~K{\"u}{\c{c}}{\"u}kyavuz.
\newblock A graph-based decomposition method for convex quadratic optimization with indicators.
\newblock \emph{Mathematical Programming}, 200\penalty0 (2):\penalty0 669--701, 2023.

\bibitem[Lov{\'a}sz(1983)]{lovasz1983submodular}
L.~Lov{\'a}sz.
\newblock Submodular functions and convexity.
\newblock In \emph{Mathematical programming the state of the art}, pages 235--257. Springer, 1983.

\bibitem[Lu and Shiou(2002)]{lu2002inverses}
T.-T. Lu and S.-H. Shiou.
\newblock Inverses of 2$\times$ 2 block matrices.
\newblock \emph{Computers \& Mathematics with Applications}, 43\penalty0 (1-2):\penalty0 119--129, 2002.

\bibitem[Plemmons(1977)]{plemmons1977m}
R.~J. Plemmons.
\newblock {M-matrix characterizations. I -- nonsingular M-matrices}.
\newblock \emph{Linear Algebra and its Applications}, 18:\penalty0 175--188, 1977.

\bibitem[Shafieezadeh-Abadeh and K{\i}l{\i}n{\c{c}}-Karzan(2024)]{shafieezadeh2023constrained}
S.~Shafieezadeh-Abadeh and F.~K{\i}l{\i}n{\c{c}}-Karzan.
\newblock Constrained optimization of rank-one functions with indicator variables.
\newblock \emph{Mathematical Programming}, 2024.
\newblock \doi{https://doi.org/10.1007/s10107-023-02047-y}.
\newblock Article in Advance.

\bibitem[Wei et~al.(2020)Wei, G{\'o}mez, and K{\"u}{\c{c}}{\"u}kyavuz]{wei2020convexification}
L.~Wei, A.~G{\'o}mez, and S.~K{\"u}{\c{c}}{\"u}kyavuz.
\newblock On the convexification of constrained quadratic optimization problems with indicator variables.
\newblock In \emph{International Conference on Integer Programming and Combinatorial Optimization}, pages 433--447. Springer, 2020.

\bibitem[Wei et~al.(2022)Wei, G\'omez, and K\"u\c{c}\"ukyavuz]{wei2021ideal}
L.~Wei, A.~G\'omez, and S.~K\"u\c{c}\"ukyavuz.
\newblock Ideal formulations for constrained convex optimization problems with indicator variables.
\newblock \emph{Mathematical Programming}, 192\penalty0 (1-2):\penalty0 57--88, 2022.

\bibitem[Wei et~al.(2024)Wei, Atamt{\"u}rk, G{\'o}mez, and K{\"u}{\c{c}}{\"u}kyavuz]{wei2023convex}
L.~Wei, A.~Atamt{\"u}rk, A.~G{\'o}mez, and S.~K{\"u}{\c{c}}{\"u}kyavuz.
\newblock On the convex hull of convex quadratic optimization problems with indicators.
\newblock \emph{Mathematical Programming}, 204\penalty0 (1-2):\penalty0 703--737, 2024.

\end{thebibliography}
\appendix

\section{Construction of true signals in computational experiments}\label{sec:construction}
Here we discuss how to construct the true signal $\{X_i\}_{i\in V}$. Note that since each node in $\mathcal{G}$ is uniquely determined by its coordinates $1\leq k, \ \ell\leq m$ in the $m\times m$ grid, we let $X_{(k,\ell)}$ denote the value of the signal at those coordinates.

The true signal is initially set to $0$ at all its coordinates. We then repeat the following process three times:
\begin{enumerate}
\item Uniformly pick coordinates $2\leq k, \ \ell\leq 9$.
\item Generate a 9-dimensional Gaussian spike $\bm{s}\sim \mathcal{N}(\bm{0},\bm{\Theta^{-1}})$, where 
$$\Theta=\begin{pmatrix} \small 4 & -1 & 0 & -1 & 0 & 0 &0 &0 &0\\
-1 & 4 & -1 & 0 & -1 & 0 &0 &0 &0\\
0 & -1 & 4 & 0 & 0 & -1 &0 &0 &0\\
-1 & 0 & 0 & 4 & -1 & 0 &-1 &0 &0\\
0 & -1 & 0 & -1 & 4 & -1 &0 &-1 &0\\
0 & 0 & -1 & 0 & -1 & 4 &0 &0 &-1\\
0 & 0 & 0 & -1 & 0 & 0 &4 &-1 &0\\
0 & 0 & 0 & 0 & -1 & 0 &-1 &4 &-1\\
0 & 0 & 0 & 0 & 0 & -1 &0 &-1 &4\\
\end{pmatrix} \cdot$$

\item For all $0\leq j_1, j_2\leq 2$, let $h =1+3j_1+j_2$ and update $$X_{(k-1+j_1,\ell-q+j_2)}\leftarrow X_{(k-1+j_1,\ell-1+j_2)}+\left|s_{h}\right|.$$
\end{enumerate}

\end{document}